%% file: THHBP2BP1.tex
\documentclass[12pt]{amsart}
\usepackage[margin=1.2in]{geometry}

\input{THHpreamble}

\title[THH of truncated Brown--Peterson spectra II]{Topological Hochschild homology of truncated Brown--Peterson spectra II} 
\author{Gabriel Angelini-Knoll}
\address{Department of Mathematics, Applied Mathematics, and Statistics, Case Western Reserve University, Cleveland OH, USA}
\email{gabriel.angelini-knoll@case.edu}
\author{Maxime Chaminadour}
\address{Sino-French Institute, Renmin University of China, Suzhou, China}
\email{maxime.chaminadour@ruc.edu.cn}

\begin{document} 
\input{abstract}     
         \maketitle 
	\tableofcontents	

\input{intro} 
	\input{recollections}
	\input{bound}
        \input{spectseqcomputation}

\input{extensions}
	\bibliographystyle{alpha}
	\bibliography{THH}
\end{document}

%% file: THHpreamble.tex
\usepackage{amsmath}
\usepackage{amsthm}
\usepackage{amssymb}
\usepackage{lscape,xcolor}
\usepackage{graphicx}
\usepackage{mathrsfs}
\usepackage{mathtools}
\usepackage{stmaryrd}
\usepackage{verbatim}
\usepackage{rotating}
\usepackage{tikz-cd}
\usepackage{amsrefs}
\usepackage[colorlinks=true, linkcolor=blue, menucolor=blue]{hyperref}
\usepackage{euscript}
\usepackage[colorinlistoftodos]{todonotes}
\usepackage{spectralsequences}
\usepackage{xy}

\usepackage{xcolor}
\definecolor{seagreen}{RGB}{46,139,87}
\definecolor{maroon}{RGB}{128,0,0}
\definecolor{darkviolet}{RGB}{148,0,211}
\definecolor{twelve}{RGB}{100,100,170}
\definecolor{thirteen}{RGB}{100,150,50}
\definecolor{fourteen}{RGB}{200,0,0}
\definecolor{fifteen}{RGB}{0,200,0}
\definecolor{sixteen}{RGB}{0,0,200}
\definecolor{seventeen}{RGB}{200,0,200}
\definecolor{eighteen}{RGB}{0,200,200}

\allowdisplaybreaks[1]

\newcommand{\mmod}{\! \sslash \!}

\newcommand{\mr}[1]{\mathrm{#1}}

\newcommand{\br}[1]{\overline{#1}}

\newcommand{\Z}{\mathbb{Z}}

\newcommand{\F}{\mathbb{F}}

\newcommand{\bE}{\mathbb{E}}
\newcommand{\bF}{\mathbb{F}}
\newcommand{\bZ}{\mathbb{Z}}

\newcommand{\taf}{\mathrm{taf}}
\newcommand{\THC}{\mathrm{THC}}

\newcommand{\tmf}{\mathrm{tmf}}

\def \HF2{\mr{H}\F_2}

\newcommand{\MU}{\mr{MU}}

\newcommand{\BP}{\mr{BP}}

\newcommand{\ku}{\mathrm{ku}}

\DeclareMathOperator{\Ext}{Ext}
\DeclareMathOperator{\Tor}{Tor}

\DeclareMathOperator{\im}{im}

\DeclareMathOperator{\map}{map}

\DeclareMathOperator{\THH}{THH}

\def \AA0{\br{A \mmod A(0)}_*}
\def \AA2{A\mmod A(2)_*}

\def \AE2{A\mmod E(2)_*}
\renewcommand{\AE}[1]{A\mmod E(#1)_*}

\newcommand{\op}{\textup{op}}

\newtheorem{theorem}[equation]{Theorem}

 \newtheorem{corollary}[equation]{Corollary}
 \newtheorem{lemma}[equation]{Lemma}
 \newtheorem{proposition}[equation]{Proposition}

 \newtheorem*{thm*}{Theorem}
 \newtheorem*{cor*}{Corollary}
 \newtheorem*{lem*}{Lemma}
 \newtheorem*{prop*}{Proposition}
  \newtheorem*{not*}{Notation}

\theoremstyle{definition}
 
 \newtheorem{example}[equation]{Example}
 
 \newtheorem{remark}[equation]{Remark}

 \newtheorem{conjecture}[equation]{Conjecture}
  \newtheorem{runningassumption}[equation]{Running Assumption}

\newtheorem*{defn*}{Definition}
\newtheorem*{ex*}{Example}
\newtheorem*{exs*}{Examples}
\newtheorem*{rmk*}{Remark}
\newtheorem*{claim*}{Claim}

\numberwithin{equation}{section}
\numberwithin{figure}{section}

%% file: abstract.tex
\begin{abstract}
We compute topological Hochschild homology of $\mathbb{E}_3$-$\MU$-algebra forms of the second truncated Brown-Peterson spectrum with Adams summand coefficients at $p=2$ and conditionally at arbitrary primes.  
We also provide a new computational tool, a variant of the Brun spectral sequence, for computing topological Hochschild homology of truncated Brown--Peterson spectra with certain coefficients. As a consequence, we show that $\mathbb{E}_3$-$\MU$-algebra forms of truncated Brown--Peterson spectra $\mathrm{BP}\langle n\rangle$ are not Thom spectra at the prime $p=2$ for any $n\ge 2$. 
\end{abstract}

%% file: intro.tex
\section{Introduction}
Topological Hochschild homology has applications to many areas of mathematics, for example deformations of $\mathbb{A}_{\infty}$-algebras~\cite{Laz01}, string topology~\cite{CJ02}, algebraic K-theory~\cite{BHM93}, $p$-adic Hodge theory~\cite{BMS19}  and knot Floer homology~\cite{LOT15}. It also played an important role in the recent disproof of the telescope conjecture~\cite{BHLS23}. The topological Hochschild homology of $\mathbb{F}_p$ can be used to provide a new proof of Bott periodicity~\cite{HN20} and variants on topological Hochschild homology have been used to provide a new proof of the Segal conjecture for cyclic groups of prime order~\cites{HW21,AKZ25}. 
We therefore consider the topological Hochschild homology of ring spectra to be of fundamental importance. 

In homotopy theory, in addition to characteristic zero and characteristic $p$, there are characteristics corresponding to pairs $(n,p)$ where $n\ge 1$ is called the height and $p$ is a prime. This also leads to higher height analogues of rings of integers in local fields of which truncated Brown--Peterson spectra are an example~\cite{Rog25}. The topological Hochschild homology of the integers $\mathbb{Z}$, and more generally rings of integers in number fields, has been related to special values of Dedekind zeta functions in work of Morin~\cite{Mor24}, which builds on work of Bhatt--Morrow--Scholze~\cite{BMS19}. To generalize this to characteristics corresponding to higher heights $n\ge 1$, a complete understanding of topological Hochschild homology of truncated Brown--Peterson spectra $\BP\langle n\rangle$ is desirable.

When $n=-1$ then $\BP\langle -1\rangle=\mathbb{F}_p$ and when $n=0$ then $\BP\langle 0\rangle=\mathbb{Z}_{(p)}$ and their topological Hochschild homology was computed by B\"okstedt~\cite{Bok85}. When $n=1$, then $\BP\langle 1\rangle=\ell$ is the Adams summand of $p$-local complex K-theory $\mathrm{ku}_{(p)}$, which is closely related to the theory of complex vector bundles. The topological Hochschild homology of $\ell$ completely computed by~\cite{AHL10} and the topological Hochschild homology of complex topological K-theory was computed by~\cite{Lee26} and~\cite{Cha26} using different methods. The topological Hochschild homology of Thom spectra is also often well-understood, see~\cite{BCS10} for example. However, we will see that $\mathrm{BP}\langle n\rangle$ is a not a Thom spectrum, at least at the prime $p=2$. Beyond these examples, there are not many complete computations of $\THH_*(R)$ where $R$ is a ring spectrum that is not Eilenberg--MacLane. 

To compute topological Hochschild homology of truncated Brown--Peterson spectra, the traditional approach is to compute the topological Hochschild homology with coefficients in a finite field $\mathbb{F}_p$ and then use a lattice of Bockstein spectral sequences to reduce down to topological Hochschild homology of truncated Brown--Peterson spectra. Here one uses the fact that the coefficients of truncated Brown--Peterson spectra are $\BP\langle n\rangle_*=\mathbb{Z}_{(p)}[v_1,\cdots ,v_n]$ and we can therefore consider $v_i$-Bockstein spectral sequences such as 
\[ 
\THH_*(\BP\langle n\rangle,\BP\langle n\rangle/v_i)[v_i]\implies \THH_*(\BP\langle n\rangle)\,.
\]

The first step is therefore the computation 
\[ \THH_*(\BP\langle n\rangle,\mathbb{F}_p)=\mathbb{F}_p[\mu^{p^{n+1}}]\langle \lambda_1,\dots ,\lambda_{n+1} \rangle
\]
first appearing in~\cite{ACH24} (cf.~\cite{Ang26}). Here  $\mathbb{F}_p\langle \lambda_1,\cdots ,\lambda_n\rangle$ denotes an exterior algebra on generators in degree $|\lambda_{i}|=2p^i-1$ for $1\le i\le n+1$ and $\mathbb{F}_p[\mu^{p^{n+1}}]$ is a polynomial algebra with generator in degree $|\mu^{p^{n+1}}|=2p^{n+1}$. From there the first author computed 
\[ \THH_*(\BP\langle n\rangle , \mathbb{Z}_{(p)})_p^{\wedge}\]
with Culver and H\"oning for arbitrary heights $n$ at the prime $p=2$ and conditionally at primes $p>2$ as well as 
$\THH_*(\BP\langle 2\rangle , M)$
when $M\in \{\BP\langle 2\rangle/(p,v_1),\BP\langle 2\rangle /(p,v_2)\}$~\cite{ACH24}. The next goal in this program is to compute $\THH_*(\BP\langle 2\rangle,\BP\langle 1\rangle)$,
which we accomplish in this paper. 

\begin{theorem}\label{thm:main}
  Suppose $p=2$ and $\mathrm{BP}\langle 2\rangle$ is a $\mathbb{E}_3$-$\mathrm{MU}$-algebra form of the second truncated Brown--Peterson spectrum or $p=3$ and $\mathrm{BP}\langle 2\rangle=\mathrm{taf}^{D}$. 
  The topological Hochschild homology of $\BP\langle 2\rangle$ is a direct sum 
  \[
    \THH_*(\BP \langle 2 \rangle , \BP \langle 1 \rangle) \cong \Z_{(p)}[v_1] \langle \sigma v_2 \rangle \oplus \mathcal{F} \oplus \Sigma^{2p - 1} (\mathcal{F}_{\geq 2p^2 - 1}) \oplus \mathcal{T} \oplus \Sigma^{2p - 1} \mathcal{T} \,,
  \]
  as $\mathbb{Z}_{(p)}$-modules, where $\mathcal{T}$ is the $\Z_{(p)}[v_1]$-submodule of $\THH_*(\BP \langle 1 \rangle)$ containing the torsion elements that are in the image of the multiplication by $v_1^p$ map,  $\mathcal{F}$ denotes the submodule of the torsion free part of $\THH_*(\BP \langle 1 \rangle)$ of elements not of the form $v_1^k \cdot 1$ for some $k \geq 0$, and $\mathcal{F}_{\geq 2p^2 -1}$ is the submodule of $\mathcal{F}$ whose elements are in degree $\geq 2p^2 -1$.
\end{theorem}
Our work builds on~\cite[Theorem~3.8]{ACH24}, which somewhat restricts the scope of our results to certain forms of $\BP\langle n\rangle$, see Running Assumption~\ref{runningassumption}. We have stated the version of the theorem that holds unconditionally above. 

As part of this program, we present a new computational tool, which is in some sense an instance of the Brun spectral sequence considered by~\cites{Bru00,Hon20}.

This computational tool gives an inductive approach to computing $\THH(\BP\langle n\rangle)$. First one computes $\THH_*(\BP\langle n-1\rangle)$ along with the cap product of a certain class in the topological Hochschild cohomology of $\BP\langle n-1\rangle$. Then one uses the Brun spectral sequence from Proposition~\ref{prop:brunss} to compute $\THH_*(\BP\langle n\rangle,\BP\langle n-1\rangle)$.  Finally, one computes the remaining $v_{n}$-Bockstein spectral sequence to compute $\THH_*(\BP\langle n\rangle)$, which we leave to future work. 

As an application, we show that $\BP\langle n\rangle$ is not a Thom spectrum of a $2$-fold loop map over the sphere spectrum for $n\ge 2$ at the prime $p=2$. When $n=-1,0$ it is known that $\mathrm{BP}\langle n\rangle$ is the Thom spectrum of a $2$-fold loop map at $p=2$ by Mahowald~\cite{Mah79}. When $n=1$, it is known that $\BP\langle 1\rangle$ is not a Thom spectrum by Mahowald~\cite{Mah87} and our argument builds on work of Angeltveit--Hill--Lawson~\cite{AHL09}, which reproved Mahowald's theorem at $n=1$ using topological Hochschild homology. 
\begin{theorem}\label{thm:main-thom}
Truncated Brown--Peterson spectra $\BP\langle n\rangle$ are not Thom spectra of $2$-fold loop maps for any $n\ge 2$ at the prime $p=2$. 
\end{theorem}

\subsection{Outline}
We first recall the necessary results from~\cite{AHL10}, \cite{AR05}, and \cite{ACH24} that we build on in this paper in Section~\ref{recollections}. In Section~\ref{sec:bounding}, we present a Brun spectral sequence for computing $\mathrm{THH}_*(\BP\langle n\rangle ,\BP\langle n-1\rangle)$. In Section~\ref{THH}, we use this spectral sequence to compute topological Hochschild homology of $\BP\langle 2\rangle$ with coefficients in $\BP\langle 1\rangle$ and we resolve  $\BP\langle 1\rangle_*$-module extensions. Finally, we prove that $\mathbb{E}_3$-$\MU$-algebra forms of the $n$-th truncated Brown--Peterson spectra are not Thom spectra at $p=2$ in Section~\ref{notThom}. 

\subsection{Conventions}
We write $L_{E}$ for Bousfield localization at a spectrum $E$. By an $
\mathbb{E}_{3}$ $\MU$-algebra form of $\BP\langle n\rangle$ we mean a $p$-local $\mathbb{E}_{3}$ $\MU$-algebra $R$ such that
the composite
\[ \mathbb{Z}_{(p)}[v_1,\cdots ,v_n]\subset \mathrm{BP}_*\subset \mathrm{MU}_{(p)}\to \pi_*R
\]
is an isomorphism. The last map in this composite is the p-localized unit of the $\mathbb{E}_3$ $\MU$-algebra
structure. Througout, we write $\BP\langle n\rangle$ for fixed family of $\bE_3$-$\MU$-algebra forms of $\BP\langle n\rangle$ in the such that there are maps
\[
\BP\to \MU \to  \dots \to \BP\langle n\rangle\to  \BP\langle n-1\rangle\to  \dots  \to \bZ_{(p)}\to \bF_{p}
\]
of $\bE_2$-rings. 
We fix classes $v_i$ such that on graded commutative rings
\[
\BP_*\to  \BP\langle n\rangle_*
\]
is given by sending $v_i$ to $v_i$ for $0\le i\le n$ (with $v_0=p$) and $v_j\mapsto 0$ otherwise. 
This also fixes the map of graded commutative rings 
\[  
    \BP\langle n\rangle_*\to  \BP\langle n-1\rangle_*
\]
sending $v_i$ to $v_i$ for $0\le i\le n-1$ and $v_n\mapsto 0$. Such a family exists by~\cite{HW22}*{Theorem~A}. 
Alternatively, when $n\le 2$ and $p\in \{2,3\}$, there are $\bE_2$-$\MU$-algebra forms of $\BP\langle 2\rangle$ denoted $\tmf_1(3)$ and $\taf^{D}$ respectively, which have the extra property that their $\bE_3$-ring structures lift to $\bE_\infty$-ring structures by~\cites{HL10,LN12,LN14,CM15}. 

Recall that if $R$ is an $\mathbb{E}_1$-ring and $M$ is an $R\otimes R^{\op}$-module, then we can define 
\[ 
\THH(R,M):=M\otimes_{R\otimes R^{\op}}R
\]
and when $M=R$ we simply write $\THH(R):=\THH(R,R)$. Here we write $\otimes$ for the derived smash product in spectra. We also write $\otimes$ for the tensor product of abelian groups and we let context dictate the meaning. We also define 
\[ 
\THC(R,M):=\map_{R\otimes R^{\op}}(M,R)
\]
where we write $\map_{R\otimes R^{\op}}$ for the mapping spectrum in $R\otimes R^{\op}$-modules. 

We write $\dot{=}$ for an equality that holds up to multiplying one side by a unit. 

\subsection{Acknowledgements}
The first author would like to thank Dominic Leon Culver and Eva H\"oning for their contributions during a previous collaboration. The first author would also like to thank Haldun \"Ozg\"ur Bayindir for helpful conversations related to this project. The authors would also like to thank Jiajun Feng for catching an error in a previous draft of this paper. Finally, the authors would like to thank the organizers of the the IWoAT workshop in 2025 in Hangzhou, China where the authors began this collaboration. 

%% file: recollections.tex
\section{Recollections}\label{recollections}

Associated to the square 
\[
\begin{tikzcd}
\BP\langle 1\rangle \ar[r] \ar[d] & H\mathbb{Z}_{(p)} \ar[d] \\ 
k(1) \ar[r] & H\mathbb{F}_p
\end{tikzcd}
\]
there is a square of Bockstein spectral sequences 
\[
\begin{tikzcd}
\THH_*(\BP\langle n \rangle ,\mathbb{F}_p)[v_0,v_1]\arrow[Rightarrow]{r} \arrow[Rightarrow]{d} &  \THH_*(\BP\langle n\rangle ,\mathbb{Z}_{(p)})[v_1]\arrow[Rightarrow]{d}\\ 
\THH_*(\BP\langle n\rangle ,k(1))[v_0]\arrow[Rightarrow]{r}  &  \THH_*(\BP\langle n \rangle ,\BP\langle 1\rangle) 
\end{tikzcd}
\]
for any integer $n\ge 1$. The first Bockstein spectral sequence in each composite of spectral sequences above was computed in~\cite{MS93},~\cite{AR05} and~\cite{AHL10} for $n=1$ and~\cite{ACH24} for $n=2$. We recall some of these results below.

\begin{theorem}[{\cite[\S~3.2]{AHL10}}]
The homotopy of $\THH(\BP\langle 1\rangle;\mathbb{Z}_{(p)})$ is a copy of $\mathbb{Z}_{(p)}$ generated by $\lambda_1$ plus torsion. The torsion is generated as a $\mathbb{Z}_{(p)}$-module by the
elements $a_i$ and $b_i$ where $a_i=\mu_2^{i-1}\lambda_2$ and $b_i=\mu_2^{i-1}\lambda_1\lambda_2$ for $i\ge 1$ and they both have order $p^{k+1}$, where $k= \nu_{p}(i)$, the
$p$-adic valuation of $i$. Here $|a_i|=2p^2i-1$ and $|b_i|=2p^2i+2p-2$
\end{theorem}

\begin{theorem}[{\cite{ACH24}*{Theorem~3.8}}]\label{thm:BP2-coeff}
Suppose the error term \cite{ACH24}*{(3--7)} vanishes at $p>3$. There is an isomorphism 
\[ 
\THH_*(\BP\langle 2\rangle ;\mathbb{Z}_{(p)})=\mathbb{Z}_{(p)}\langle \lambda_1 ,\lambda_2\rangle \otimes \left ( \mathbb{Z}_{(p)}\oplus T_0^2 \right )
\]
where 
\[ T_0^2 = \oplus_{s\ge 1} \mathbb{Z}/p^s\otimes \mathbb{Z}_{(p)}[\mu_3^{p^s}]\otimes \mathbb{Z}_{(p)}\{\lambda_{s+2}\mu_3^{jp^{s-1}} \mid 0\le j\le p-2 \} \,. 
\]
Here $|\mu_3|=2p^3$ and $|\lambda_{s}|=2p^{s}-1$. 
\end{theorem}
Note that the error term \cite{ACH24}*{(3--7)} vanishes for $\mathbb{E}_\infty$ algebras so the theorem holds unconditionally for $\BP\langle 2\rangle=\taf^{D}$ from~\cite{HL10}. 

The map 
\[ 
  \THH(\BP\langle 2\rangle ; \mathbb{F}_p)\to  \THH(\BP\langle 1\rangle ; \mathbb{F}_p) 
\]
sends $\lambda_s$ to $\lambda_s$ for $s=1,2$ and $\mu_3$ to $\mu_2^p$. For consistency, let $a_{pi+1}=\lambda_1\mu_3^{i}$ and $b_{pi+1}=\lambda_1\lambda_2\mu_3^{i}$ 

\begin{corollary}
  The map 
  \[ 
    \THH_*(\BP\langle 2\rangle ;\mathbb{Z}_{(p)})\longrightarrow \THH_*(\BP\langle 1\rangle ;\mathbb{Z}_{(p)})
  \]
  sends $\lambda_1$ to $\lambda_1$, and it sends $\lambda_1\mu_3^i$ to $a_{pi+1}$ and $\lambda_1\lambda_2\mu_3^i$ to $b_{pi+1}$. 

  The map 
  \[ 
    \THH_*(\BP\langle 2\rangle ;k(1))\longrightarrow \THH_*(\BP\langle 1\rangle ;k(1))
  \]
  sends $x_{n,m}$ to $x_{n,m}$ and $x'_{n,m}$ to $x'_{n,m}$ whenever this makes sense.
\end{corollary}

We recall the structure of $\THH_*(\BP \langle 1 \rangle) = \THH_*(\ell)$.

\begin{theorem}[{\cite[Sections 6.2--6.3]{AHL10}}]\label{thm:AHL}
  $\THH_*(\BP \langle 1 \rangle)$ is a quotient of the $\Z_{(p)}[v_1]$-module
  \begin{multline}
    \Z_{(p)}[v_1]\{1,\, \lambda_1,\, v_0^n a_{p^n},\,n\geq 0\} 
    \oplus \Z_{(p)}[v_1]\{v_0^h b_{\alpha p^n},\, n\geq 0,\, \alpha\geq 1,\, p\nmid \alpha,\, h \geq 0\}
  \end{multline}
  by the relations in the torsion-free part:
  \begin{itemize}
  \item $p\cdot a_1 = v_1^p \lambda_1$,
  \item $p\cdot v_0^n a_{p^n} = v_1^{p^{n+1}}v_0^{n-1} a_{p^{n-1}}$ for any $n\geq 1$,
  \end{itemize}
  and the relations in the torsion part:
  \begin{itemize}
  \item $v_0^{h} b_{\alpha p^n} = 0$ for any $\alpha \geq 1$ and $n \geq 0$, $\alpha$ not divisible by $p$, and $h \geq n + 1$,
  \item $v_1^{p^{n-h+1}+p^{n-h}+\dots+p }\cdot v_0^h b_{\alpha p^n} = 0$ for any $\alpha \geq 1$ and $n \geq 0$, $\alpha$ not divisible by $p$ and  $0 \leq h \leq n$,
  \item $p\cdot b_{(\beta p+p-1)p^n} = v_0 b_{(\beta p+p-1)p^n}+v_1^{p^{n+2}}v_0^{\nu_p(\beta)} b_{\beta p^{n+1}}$ for any $\beta \geq 1$ and $n \geq 0$. Here $\nu_p$ denotes $p$-adic valuation.
  \item $p\cdot v_0^h b_{\alpha p^n} = v_0^{h+1} b_{\alpha p^n}$ for any $\alpha \geq 1$, $n \geq 0$, $\alpha$ not divisible by $p$, and any $1 \leq h \leq n$, or $h=0$ not in the previous case.
  \end{itemize}
  Here $|\lambda_1|=2p-1$, $|v_0^na_{p^n}|=2p^{n+2}-1$, $|v_1|=2p-2$ and $|v_0^hb_{\alpha p^n}|=2p^{n+2}\alpha+2p-2$.

  We write $\THH_*(\BP \langle 1 \rangle) = \Z_{(p)}[v_1]\{1 \} \oplus \mathcal{F}_1 \oplus T_1^1$ where $\Z_{(p)}[v_1]\{1\} \oplus \mathcal{F}_1$ is the torsion-free part and $T_1^1$ is the torsion part.
\end{theorem}

We also recall the computation of the rationalization of topological Hochschild homology of truncated Brown--Peterson spectra. 

\begin{proposition}[{\cite{ACH24}*{Proposition~3.7}}]\label{prop:rational}
Let $0\le m\le n$. There is a preferred isomorphism 
\[ 
\mathbb{Q}[v_1,\cdots ,v_m]\langle\sigma v_1,\sigma v_2,\cdots ,\sigma v_n) \cong \THH_*(\BP\langle n\rangle ; \BP\langle m\rangle)_{\mathbb{Q}} 
\]
where we write $X_{\mathbb{Q}}$ for the Bousfield localization at the Eilenberg--MacLane $H\mathbb{Q}$. 
\end{proposition}

\begin{proposition}\label{prop:K1local}
Let $p\ge 3$, then we have an equivalence 
\[ 
\mathrm{BP}\langle 2\rangle[v_1^{-1}]\vee \Sigma^{2p-1}\mathrm{BP}\langle 2 \rangle[v_1^{-1}]\overset{\simeq}{\longrightarrow} L_{K(1)}\mathrm{THH}(\BP\langle 2\rangle )
\]
and an equivalence 
\[ 
\mathrm{BP}\langle 1\rangle[v_1^{-1}]\vee \Sigma^{2p-1}\mathrm{BP}\langle 1 \rangle[v_1^{-1}]\overset{\simeq}{\longrightarrow} L_{K(1)}\mathrm{THH}(\BP\langle 2\rangle ; \BP\langle 1\rangle)
\]
\end{proposition}
\begin{proof}
This follows from \cite[Theorem~4.6]{ACH24} (cf.~\cite[Corollary~5.5]{AR20}). 
\end{proof}
\begin{remark}\label{rem:p-torsion-and-v1-torsion}
Since $\THH_*(\BP\langle n\rangle;\mathbb{F}_p)$ is finitely generated as an $\mathbb{F}_p$-module, we know that $\THH_*(\BP\langle n\rangle)$ is finitely generated as a $\BP\langle n\rangle_*$-module. Suppose 
\[M_1\subset \THH_*(\BP\langle 2\rangle,\BP\langle 1\rangle) \]
such that for each $x\in M_1$  there exists an integer $n(x)\ge 0$ such that $p^{n(x)}x=0$ and $v_1^kx'\ne 0$ for all $k\ge 1$ for some $x'\in M_1$
or 
\[ 
M_2\subset \THH_*(\BP\langle 2\rangle,\BP\langle 1\rangle) 
\]
such that for each $y\in M_2$ there exists an $n(y)\ge 0$ such that $v_1^{n(y)}y=0$ and $p^{j}y'\ne 0$ for all $j\ge 1$ for some $y'\in M_2$ 
then we would have a submodule 
\[
\mathbb{Z}/p^{n(x')}[v_1^{\pm 1}]\{x'\} \subset \pi_*L_{K(1)}\THH(\BP\langle 2\rangle,\BP\langle 1\rangle)\]
or 
\[
\mathbb{Q}[v_1]/(v_1^{n(y')})\{y'\}\subset \pi_*\THH(\BP\langle 2\rangle,\BP\langle 1\rangle)_{\mathbb{Q}} \,.\] 
and, as a consequence of Proposition~\ref{prop:K1local} and Proposition~\ref{prop:rational}, these types of submodules do not occur by the classification of finitely generated modules over a graded PID~\cite[Theorem~2.1]{CZ05}. 
\end{remark}

%% file: bound.tex
\section{Bounding Hochschild homology of \texorpdfstring{$\BP\langle n\rangle$}{BP n }}\label{sec:bounding}
The goal of this section is to prove the following result.  
\begin{proposition}\label{prop:brunss}
There is a multiplicative Brun spectral sequence 
\begin{equation}\label{BrunSS}
\THH_*(\BP\langle n-1 \rangle)\langle \sigma v_{n}\rangle \implies \THH_*(\BP\langle n\rangle,\BP\langle n-1\rangle)\,.
\end{equation}
which is a spectral sequence of $(\BP\langle n-1 \rangle\otimes_{\BP\langle n\rangle}\BP\langle n-1 \rangle)_*$-algebras and the differentials are $(p,v_1,\cdots ,v_{n-1})$-linear and $\sigma v_{n}$-linear. The $d_1$-differential is given by the cap product of a certain class in $\THC^{2p^{n+1}}(\BP\langle n-1\rangle)$.  
\end{proposition}

To prove this, we first need the following lemma. 
\begin{lemma}\label{lem:BPn-rel-cooperations}
Suppose $-1\le m\le n$. We compute that 
\[\pi_*\BP\langle m \rangle\otimes_{\BP\langle n\rangle}\BP\langle m \rangle=\mathbb{Z}_{(p)}[v_1,\cdots ,v_{m}]\langle \sigma v_{m+1} , \cdots ,\sigma v_{n} \rangle \,. \]
when $m \geq 0$ and 
\[
\pi_*\BP\langle -1 \rangle\otimes_{\BP\langle n\rangle}\BP\langle -1\rangle=\mathbb{F}_p\langle \sigma v_{m+1} , \cdots ,\sigma v_{n} \rangle \,. 
\]
\end{lemma}
\begin{proof}
This follows directly from the K\"unneth spectral sequence
\[ 
\Tor_*^{\BP\langle n\rangle_*}(\pi_*\BP\langle m\rangle,\pi_*\BP\langle m\rangle)\Longrightarrow \pi_*\BP\langle m\rangle\otimes_{\BP\langle n\rangle}\BP\langle m\rangle
\]
which collapses at the $\mathrm{E}_2$-page because all differentials on algebra generators land in trivial bidegrees and $\mathbb{Z}_{(p)}[v_1,\cdots ,v_{m}]$ splits off. 
\end{proof}

By the same argument as \cite{AHL10}*{Proposition~5.2}, we compute that 
\[ \pi_*H\mathbb{F}_p\otimes_{\BP}\BP\langle n\rangle=\mathbb{F}_p\langle \overline{\tau}_{n+1},\overline{\tau}_{n+2},\cdots \rangle\]
and using the universal coefficient spectral sequence we deduce that 
\[ \THC^*(\BP\langle n\rangle/\BP ; \mathbb{F}_p)=\mathbb{F}_p[e_{n+1},e_{n+2},\cdots ]\]
where $\varepsilon_i$ corresponds to $\tau_i$. Since these classes are in even degrees, the Bockstein spectral sequences collapse leading to
\[ 
\THC^*(\BP\langle n\rangle/\BP )=\BP\langle n\rangle_*[[e_{n+1},e_{n+2},\cdots]]
\]
The computation that 
\[ 
\THC^*(\BP\langle n\rangle ; \mathbb{F}_p)=\mathbb{F}_p\langle \lambda_1^{\vee},\cdots ,\lambda_{n+1}^{\vee})\otimes \Gamma ((\mu^{p^{n+1}})^{\vee})
\]
is immediate from~\cite{ACH24}*{Proposition~2.9}. By examination of the map of universal coefficient spectral sequences, we can also deduce that the map
\[ \THC^*(\BP\langle n\rangle/\BP)\to \THC^*(\BP\langle n\rangle)\]
sends $e_{n+1}$ to $(\mu^{p^{n+1}})^{\vee}$. and consequently, there exists a lift $e_{n+1}\in \THC^{2p^{n+1}}(\BP\langle n\rangle)$ of $(\mu^{p^{n+1}})^{\vee}$ by the commuting diagram
\[
\begin{tikzcd}
\THC^*(\BP\langle n\rangle/\BP) \arrow{d} \arrow{r}& \THC^*(\BP\langle n\rangle/\BP;\mathbb{F}_p) \arrow{d}\\
\THC^*(\BP\langle n\rangle)\arrow{r} & \THC^*(\BP\langle n\rangle;\mathbb{F}_p) \,.
\end{tikzcd}
\]

\begin{proof}[Proof of Proposition~\ref{prop:brunss}]
Let $F$ denote the fiber of the map 
\[ 
\BP\langle n-1\rangle\otimes_{\BP\langle n\rangle}\BP\langle n-1\rangle \longrightarrow \BP\langle n-1\rangle \otimes_{\BP\langle n-1\rangle}\BP\langle n-1\rangle =\BP\langle n-1\rangle 
\]
of right $\mathrm{BP}\langle n-1\rangle \otimes \mathrm{BP}\langle n-1\rangle^{\op}$-modules and left $\BP\langle n-1\rangle\otimes_{\BP\langle n\rangle}\BP\langle n-1\rangle$-modules. The right $\mathrm{BP}\langle n-1\rangle \otimes \mathrm{BP}\langle n-1\rangle^{\op}$-module structure can also be considered as a $\mathrm{BP}\langle n-1\rangle$-bimodule structure. As a left $\mathrm{BP}\langle n-1\rangle$-module $F\simeq \mathrm{BP}\langle n-1\rangle$. To see this, note that $\mathrm{BP}\langle n-1 \rangle=\mathrm{BP}\langle n\rangle/v_n$ and therefore 
\[ \mathrm{BP}\langle n-1 \rangle\otimes_{\mathrm{BP}\langle n\rangle} \mathrm{BP}\langle n\rangle/v_n=\mathrm{BP}\langle n-1 \rangle/v_n\]
but $v_n$ acts trivially on $\mathrm{BP}\langle n-1 \rangle$ so $\mathrm{BP}\langle n-1 \rangle/v_n\simeq \mathrm{BP}\langle n-1 \rangle\vee \Sigma^{2p^{n-1}}\mathrm{BP}\langle n-1 \rangle$ as a left $\mathrm{BP}\langle n-1 \rangle$-module so the fiber can be identified with $\mathrm{BP}\langle n-1 \rangle$ as a left $\mathrm{BP}\langle n-1 \rangle$-module. A similar argument with right 
$\mathrm{BP}\langle n\rangle$-module structures implies that the fiber is equivalent to $\mathrm{BP}\langle n-1\rangle$ as a right $\mathrm{BP}\langle n-1\rangle$-module and consequently as a $\mathrm{BP}\langle n-1\rangle\otimes \mathrm{BP}\langle n-1\rangle^{\op}$-module. Note that we are not claiming that $\mathrm{BP}\langle n-1\rangle \otimes_{\mathrm{BP}\langle n\rangle}\mathrm{BP}\langle n-1\rangle$ splits as a $\mathrm{BP}\langle n-1\rangle \otimes_{\mathrm{BP}\langle n\rangle}\mathrm{BP}\langle n-1\rangle^{\op}$ bimodule. 

However, we do produce a fiber sequence 
\[ 
\Sigma^{2p^{n}-1}\THH(\BP\langle n-1 \rangle)\to  \THH( \mathrm{BP}\langle n \rangle;\mathrm{BP}\langle n-1 \rangle)\longrightarrow \THH(\BP\langle n-1 \rangle)
\]
of left $\BP\langle n-1\rangle \otimes_{\BP\langle n \rangle }\BP\langle n-1\rangle $-modules 
and the long exact sequence associated to this fiber sequence may be regarded as spectral sequence, which we call the Brun spectral sequence following~\cite{Hon20}. In fact, since the map 
\[ \mathrm{BP}\langle n-1\rangle \to \Sigma^{2p^{n}}\mathrm{BP}\langle n-1\rangle 
\]
is a $\mathrm{BP}\langle n-1\rangle \otimes \mathrm{BP}\langle n-1\rangle^{\op}$-module map, it is an element in topological Hochschild cohomology $\mathrm{THC}(\mathrm{BP}\langle n-1\rangle)$ and therefore it is a $\mathbb{E}_1$-derivation so the associated spectral sequence is multiplicative. The $d_1$-differential is therefore given by capping with this class, which can be identified with a choice of lift $\varepsilon_{n+1}$ of the class $(\mu^{p^{n+1}})^{\vee}$. 
\end{proof}
\begin{remark}
Given the equivalence  
\[ \THH(A;B)\simeq \THH(B;B\otimes_{A}B)\]
one can filter $B\otimes_{A}B$ by its Whitehead filtration $\tau_{\ge \bullet}$ to produce a spectral sequence. Such a spectral sequence was originally considered by Brun~\cite{Bru00} in the setting of FSP's and was extended to the setting of commutative ring spectra by H\"oning~\cite{Hon20}, for applications see also~\cite{Cha26} and~\cite{Hys26}. Although we consider a more general filtration of $B\otimes_{A}B$ above, we were inspired by the work of Brun and H\"oning and therefore we still call this a Brun spectral sequence. 
\end{remark}
\begin{example}
Consider the case $n=0$. Then there is a spectral sequence 
\[ 
\THH_{*}(\mathbb{F}_{p})\otimes \Lambda (\sigma v_{0})\implies \THH_{*}(\mathbb{Z}_{(p)};\mathbb{F}_{p}) \,.
\]
This spectral sequence has a differential $d_{1}(\mu)=\sigma v_{0}$ and no further differentials except those generated by the Leibniz rule yielding the known answer $\mathbb{F}_p[\mu^p]\langle \sigma v_0 \mu^{p-1}\rangle$. 
Already, this shows that although $\mathbb{F}_p\otimes_{\mathbb{Z}_{(p)}}\mathbb{F}_p$ splits as an $\mathbb{F}_p$-module it does not split as a $\mathbb{F}_p\otimes \mathbb{F}_p$-module. This is perhaps not surprising because $\pi_*\mathbb{F}_p\otimes_{\mathbb{Z}_{(p)}}\mathbb{F}_p=\Lambda (\sigma v_0)$ and the $k$-invariant $\mathbb{F}_p\to \Sigma^{2}\mathbb{F}_p$ is the non-trivial class in degree $2$ in topological Hochschild cohomology of $\mathbb{F}_p$ which is dual to $\mu$ (up to a unit).
\end{example}

\begin{example}
Consider the case $n=1$. Then there is a spectral sequence 
\[ 
\THH_{*}(\mathbb{Z}_{(p)})\otimes \Lambda (\sigma v_{1})\implies \THH_{*}(\ell;\mathbb{Z}_{p}) \,.
\]
We know $\sigma v_1$ survives because it generates a torsion free class and $\THH_*(\mathbb{Z})$ is torsion above degree zero. Specifically, 
\[
\THH_*(\mathbb{Z}_{(p)})=\begin{cases} \mathbb{Z}_{(p)} & \text{ if } n=0 \\ 
\mathbb{Z}/p^{\nu_p(k)+1}\{\lambda_1\mu_1^{k-1}\} & \text{ if } n=2pk-1>0 \\ 
0 & \text{ otherwise.}
\end{cases}
\]
By comparing to~\cite{AHL10} for example, we can determine differentials 
\[
d_1(\lambda_1\mu^{k})\dot{=}p^{\nu_p(k)}\lambda_1\mu^{k-1}\sigma v_1 
\]
for all integers $k$. 
In particular, when $k=ip-1$ for some integer $i\ge 0$ then 
\[ 
\ker
(d_1 :\mathbb{Z}/p^{\nu_p(i)+2}\{\lambda_1\mu_1^{ip-1}\}\to \mathbb{Z}/p\{\lambda_1\mu_1^{ip-2}\sigma v_1) = \mathbb{Z}/p^{\nu_p(i)+1}\{\lambda_1\mu^{ip-1}\} \,.
\]
and when $k=pi$ for some integer $i\ge 0$ then 
\[ 
\textup{coker} 
(d_1 :\mathbb{Z}/p\{\lambda_1\mu_1^{ip}\}\to \mathbb{Z}/p^{\nu_p(i)+2}\{\lambda_1 \mu_1^{ip-1}\sigma v_1 \}) = \mathbb{Z}/p^{\nu_p(i)+1}\{\sigma v_1\lambda_1\mu_1^{ip-1}\}
\]

This accounts for the classes $a_i=p\lambda_1\mu_1^{ip-1}$ and $b_i=\sigma v_1\lambda_1\mu_1^{ip-1}$ respectively. Otherwise the kernel and cokernel are trivial. In this case, the $k$-invariant is a map 
\[ H\mathbb{Z}\to \Sigma^{2p} H\mathbb{Z}\]
of $ H\mathbb{Z}\otimes H\mathbb{Z}^{\op}$-modules, corresponding to a generator of $\mathrm{THC}^{2p}(\bZ)=\mathbb{Z}/p$. Here we use the equivalence 
\[ 
\map_{H\bZ\otimes H\bZ^{\op}}(H\bZ,H\bZ)\simeq \map_{H\bZ}(\THH(\bZ),H\bZ) \,,
\]
the universal coefficient spectral sequence 
\[ 
\Ext_{\mathbb{Z}}^{*,*}(\THH_*(\mathbb{Z}),\mathbb{Z})\implies \THC^*(\mathbb{Z})
\]
associated to the equivalence, and the known computation 
\[
    \THH_n(\mathbb{Z})=\begin{cases} \mathbb{Z} &n=0 \\ 
    \mathbb{Z}/(k-1) &n=2k-1 \\ 
    0 & \text{otherwise}
    \end{cases}
\]
to produce the identity $\mathrm{THC}^{2p}(\mathbb{Z})=\mathbb{Z}/p$
\end{example}

%% file: spectseqcomputation.tex
\section{Hochschild homology  \texorpdfstring{$\BP\langle 2\rangle$}{BP2} with \texorpdfstring{$\BP\langle 1\rangle$}{BP1} coefficients}\label{THH}
This section is devoted to our computation of topological Hochschild homology of $\BP\langle 2\rangle$ with coefficients in $\BP\langle 1\rangle$. In Section~\ref{sec:SS}, we compute the spectral sequence \eqref{BrunSS} in the case $n=2$. In Section~\ref{sec:extensions}, we resolve extensions in this spectral sequence. 

Since our proof builds on Theorem~\ref{thm:BP2-coeff}, we include the following running assumption throughout the rest of the paper. 
\begin{runningassumption}\label{runningassumption}
For the remainder of this paper, we assume $p=2$ or the error term of ~\cite[Equation (3-7)]{ACH24} vanishes.
\end{runningassumption}
\begin{remark}
We expect that Running Assumption~\ref{runningassumption} is not strictly necessary and our results hold in full generality. Note that  Running Assumption~\ref{runningassumption} holds for $\mathbb{E}_\infty$-rings so it holds at $p=3$ when $\BP\langle 2\rangle=\tmf^{D}$ from~\cite{HL10}. 
\end{remark}

\subsection{Computating the spectral sequence}\label{sec:SS}

We first use the known computation of $\THH_*(\BP\langle 2\rangle;H\mathbb{Z}_{(p)})$ to deduce differentials that will be useful for bootstrapping to the case of $\THH_*(\BP\langle 2\rangle;\BP\langle 1\rangle)$. 

\begin{proposition}
  There is a spectral sequence
  \[
    \THH_*(\BP\langle 1\rangle; \Z_{(p)}) \langle \sigma v_2 \rangle \Rightarrow \THH_*(\BP \langle 2 \rangle;\Z_{(p)}) \,.
  \]
The differentials in this spectral sequence are given for any $i \geq 1$ by
  \[
    \begin{gathered}
      d_1(a_i) \dot{=} p^{\nu_p(i-1)} a_{i-1} \sigma v_2 \,, \\
      d_1(b_i) \dot{=} p^{\nu_p(i-1)} b_{i-1} \sigma v_2 \,.
    \end{gathered}
  \]
\end{proposition}

\begin{proof}
  The fiber sequence 
  \[ 
    \Sigma^{2p^2-1}\THH(\BP\langle 1\rangle)\to \THH(\BP\langle 2\rangle ;\BP\langle 1\rangle)\to \THH(\BP\langle 1\rangle)
  \]
  is of $\BP\langle 1\rangle \otimes_{\BP \langle 2 \rangle} \BP\langle 1\rangle^\op$-modules, so is compatible with the left $v_1$-multiplication map. This produces a fiber sequence
  \[
    \Sigma^{2p^2-1}\THH(\BP\langle 1\rangle; \Z_{(p)})\to \THH(\BP\langle 2\rangle ; \Z_{(p)})\to \THH(\BP\langle 1\rangle; \Z_{(p)})
  \]
  and a spectral sequence
  \[
    \THH_*(\BP\langle 1\rangle; \Z_{(p)}) \langle \sigma v_2 \rangle \Rightarrow \THH_*(\BP \langle 2 \rangle; \Z_{(p)}) \,.
  \]

  By \cite{ACH24}*{Theorem~3.18}, we know that 
  \[
    \THH_*(\BP\langle 1\rangle,\mathbb{Z}_{(p)})=\mathbb{Z}_{(p)}\langle \lambda_1\rangle \otimes (\mathbb{Z}_{(p)}\oplus T_0^1)
  \]
  and 
  \[
    \THH_*(\BP\langle 2\rangle,\mathbb{Z}_{(p)})=\mathbb{Z}_{(p)}\langle \lambda_1,\lambda_2\rangle \otimes (\mathbb{Z}_{(p)}\oplus T_0^2)\,.
  \]
  So we know that there is a long exact sequence 
  \[ 
    \to T_0^{2} \to T_0^1\overset{\partial}{\to} \Sigma^{2p^2}T_0^1 \to 
  \]
  where the connecting map $\partial$ corresponds to the differential of the spectral sequence.
  Since 
  \[ 
    T_0^1=\bigoplus_{s\ge 1}\mathbb{Z}/p^s[(\mu^{p^2})^{p^s}]\{\lambda_{1+s}(\mu^{p^2})^{jp^{s-1}}: 0\le j\le p-2\}
  \]
  and 
  \[ 
    T_0^2=\bigoplus_{s\ge 1}\mathbb{Z}/p^s[(\mu^{p^3})^{p^s}]\{\lambda_{2+s}(\mu^{p^3})^{jp^{s-1}}: 0\le j\le p-2\}
  \]
  where 
  \[ 
    \lambda_{s}=\begin{cases} \lambda_s & \text{ if } 1\le s\le n+1  \\ 
      \lambda_{s-1}(\mu^{p^{n+1}})^{p^{s-(n+2)}(p-1)} & \text{ if } s>n+1 \,.
    \end{cases}
  \]
  
  The map $T_0^2 \rightarrow T_0^1$ has image $p \cdot T_0^1$, so that each generator in degree $*$ of $T_0^1$ must have image by $\partial$ the top of the multiplication-by-$p$-tower of the generator in degree $*-2p^2$ in $T_0^1$. This is exactly the formulas we claimed for the torsion. The differentials in the non-torsion part are zero for degree reasons.
\end{proof}

\begin{lemma} \label{lem:thhell1}
  In $\THH_*(\BP\langle 1\rangle)$, if $|\alpha| = |b_i|$ for some $i \geq 1$, and $\alpha \neq 0$, then $v_1^{p-1} \alpha \neq 0$; if moreover $v_1^{p+1} \alpha = 0$ then $v_1^p \alpha = 0$ and $\alpha \in \Z_{(p)} \{ v_0^{\nu_p(i)}b_i \}$. If $\alpha$ is a multiple of $v_1$, then $\alpha \in p\Z_{(p)} \{ v_0^{\nu_p(i)}b_i \}$ 
\end{lemma}

\begin{proof}
  Since $\THH_{|b_i|}(\BP\langle 1\rangle)$ for some $i\ge 1$ is generated as a $\Z_{(p)}$-module by the element $v_1^k v_0^h b_j$ for some $h,\, k \geq 0$ and $j \geq 1$, it is sufficient to consider the case where $\alpha = v_1^k v_0^h b_j$. We consider the degrees:
  \[
    \begin{gathered}
      |\alpha| = |b_i| \\
      \iff 2j p^2 - 1 + 2p - 1 + 2k(p-1) = 2i p^2 - 1 + 2p - 1 \\
      \iff (i - j)p^2 = k(p -1)
    \end{gathered}
  \]
  so that $p^2$ divides $k$; write $k = p^2 k'$. However, the order for multiplication by $v_1$ of $v_0^h b_j$ is given by the relation
  \[
    v_1^{p^{\nu_p(j) - h + 1}+ \dots + p} v_0^h b_j = 0
  \]
  by Theorem~\ref{thm:AHL} so that since $\alpha \neq 0$,
  \[
    k = p^2 k' \leq p^{\nu_p(j) - h + 1}+ \dots + p
  \]
  and thus
  \[
    k + p - 1 \leq p^{\nu_p(j) - h + 1}+ \dots + p
  \]
  which means that $v_1^{p - 1} \alpha \neq 0$.

  If moreover $v_1^{p+1} \alpha = 0$, we have 
  \[
    |v_1^{p+1} \alpha| = |v_1^{p+1}| + |b_i| = |b_{i+1}| - 2
  \]
  but the $v_1$-Bocksteins are supported by the classes $a_j$ of $\THH_*(\BP\langle 1\rangle; \Z_{(p)})$ by \cite[Theorem~6.4]{AHL10}, whose degrees are of the form
  \[
    |a_j| = |b_j| - 2p + 1
  \]
  and thus it must be the case that $v_1^p \alpha = 0$.
  
  Finally, this means that the $v_1$-Bocksteins that kill $v_1^p \alpha$ and $v_1^p v_0^{\nu_p(i)} b_i$ are supported by some multiple of $p$ of the same class $a_{i + 1}$, so that $v_0^{\nu_p(i)} b_i$ is at the bottom of a tower of extensions that includes $\alpha$, and we conclude that  $\alpha \in \Z_{(p)}\{ v_0^{\nu_p(i)} b_i \}$. The last statement follows from the relations in Theorem~\ref{thm:AHL}. 
\end{proof}

\begin{lemma} \label{lem:thhell2}
  In $\THH_*(\BP\langle 1\rangle)$, if
  \[
    |v_0^hb_j| \leq | b_i| < | v_1^{p^{\nu_p(j) - h + 1} + \dots + p} v_0^hb_j |
  \]
  for some $ 0 \leq j \leq i$ and $0 \leq h \leq \nu_p(j)$ then
  \[
    | v_1^{p^{\nu_p(i) + 1} + \dots + p} b_i | \leq | v_1^{p^{\nu_p(j) - h + 1} + \dots + p} v_0^hb_j | \,.
  \]
\end{lemma}

\begin{proof}
  We have
  \[
    \begin{aligned}
      | v_1^{p^{\nu_p(j) - h + 1} + \dots + p} v_0^hb_j |
      & = |b_j| + 2(p^{\nu_p(j) - h + 2} - p) \\
      & = |b_{j + p^{\nu_p(j) - h}} | - 2p
    \end{aligned}
  \]
  so that the hypothesis implies $j \leq i < j + p^{\nu_p(j) - h}$, and thus $\nu_p(i) < \nu_p(j) - h$. Considering the base-$p$ expansion of $i$ and $j$, we see that $ i + p^{\nu_p(i)} \leq j + p^{\nu_p(j) - h}$ which implies the result.
\end{proof}

\begin{proposition}\label{theorem:differentialellcoeff}
The differentials in the spectral sequence
\[ 
\THH_*(\BP\langle 1\rangle)\langle \sigma v_2\rangle \implies \THH_*(\BP\langle 2\rangle,\BP\langle 1\rangle)
\]
are given by 
\[
  d_1(b_i)\dot{=} v_0^{\nu_p(i-1)}b_{i-1} \sigma v_2
\]
for each $i \geq 2$ and they are $v_1$-linear.
\end{proposition}

\begin{proof}
  First, we consider the rationalization of the fiber sequence 
  \[ 
    \Sigma^{2p^2-1}\THH(\BP\langle 1\rangle)\to \THH(\BP\langle 2\rangle ;\BP\langle 1\rangle)\to \THH(\BP\langle 1\rangle) \,.
  \]
  This produces a spectral sequence 
  \[ 
    \THH_*(\BP\langle 1\rangle)\otimes \mathbb{Q}\langle \sigma v_2\rangle \implies \THH_*(BP\langle 2\rangle,\BP\langle 1\rangle)\otimes \mathbb{Q}
  \]
  which we know collapses by Proposition~\ref{prop:rational}. The map of spectral sequences 
  \[ 
    \begin{tikzcd}
      \THH_*(\BP\langle 1\rangle)\langle \sigma v_2\rangle \ar[d] \arrow[=>]{r} & \THH_*(BP\langle 2\rangle,\BP\langle 1\rangle)\ar[d] \\
      \THH_*(\BP\langle 1\rangle)\otimes \mathbb{Q}\langle \sigma v_2\rangle \arrow[=>]{r} & \THH_*(BP\langle 2\rangle,\BP\langle 1\rangle)\otimes \mathbb{Q}
    \end{tikzcd}
  \]
  is injective on the torsion free part and therefore the torsion free summands in the top spectral sequence consist of permanent cycles.

  The map of spectral sequences induced by $\BP\langle 1\rangle\to H\mathbb{Z}_{(p)}$
allows us to determine the differentials modulo $v_1$, so that
  \[
    d_1(b_i) \dot{=} (v_0^{\nu_p(i-1)}b_{i-1} + \alpha ) \sigma v_2
  \]
  where $\alpha$ is a multiple of $v_1$. To prove our claim, it is sufficient to prove that $\alpha \in p \Z_{(p)}\{ v_0^{\nu_p(i-1)}b_{i-1} \}$, so that $v_0^{\nu_p(i-1)}b_{i-1} + \alpha \dot{=} v_0^{\nu_p(i-1)}b_{i-1}$ and thus the claimed formula  is true up to a different unit.

  Since
  \[
    v_1^{p^{\nu_p(i)} + \dots + p } b_i = 0
  \]
  it must also be the case that
  \begin{equation} \label{eq:orderalpha}
    v_1^{p^{\nu_p(i)} + \dots + p } \alpha = 0 \,.
  \end{equation}
  Let us write $k$ for the smallest integer such that $v_1^k \alpha = 0$. There are two possibilities:
  \begin{itemize}
  \item if $ |\alpha| < |b_i| \leq |v_1^k \alpha|$ then by Lemma~\ref{lem:thhell2} we have $|v_1^{p^{\nu_p(i)} + \dots + p } b_i| \leq |v_1^k \alpha|$. However, this implies that $|v_1^{p^{\nu_p(i)} + \dots + p } \alpha | < |v_1^k \alpha|$ so this is incompatible with equation \eqref{eq:orderalpha}.
  \item Otherwise, we have $|v_1^k \alpha | < |b_i|$. We know that $|\alpha | = |b_{i-1}|$ and that $|v_1^{p+1} b_{i-1}| < |b_i| < |v_1^{p+2} b_{i-1}|$ so that $k \leq p +1$. From Lemma~\ref{lem:thhell1}, we see that $v_1^p\alpha = 0$ and that $\alpha \in p\Z_{(p)}\{ v_0^{\nu_p(i-1)}b_{i-1} \}$.
  \end{itemize}
\end{proof}

\begin{remark}
The computation of the differential in Proposition~\ref{theorem:differentialellcoeff} can be considered as a computation of the action of the cap product by the class $e_{3}$ from Section~\ref{sec:bounding}, which lifts the class dual to $\mu_3\in \THH_{2p^3}(\BP\langle 2\rangle;\mathbb{F}_p)$. This cap product is computed explicitly on all classes modulo $v_1$ in~\cite[Corollary~5.7]{AHL10} and we combine this result with $v_1$-Bockstein spectral sequence information to produce the result. An explicit computation of the cap product on all of $\THH_*(\ell)$ does not appear in~\cite{AHL09} although they use certain formulas to produce hidden extensions. 
\end{remark}

We would now like to give a nice description of the kernel and cokernel of the $d_1$-differential. For this, we need the following lemma. 

\begin{lemma} \label{lem:thhell3}
  In $\THH_*(\BP\langle 1\rangle)$, $b_i$ and $v_0^{\nu_p(i-1)} b_{i-1}$ have the same order for multiplication by $p$ for any $ i \geq 2$ which is not a power of $p$. When $i = p^k$ for some $k \geq 1$, the order of $b_i$ is one more than that of $v_0^{\nu_p(i-1)} b_{i-1} = b_{i-1}$.
\end{lemma}

\begin{proof}
  If $p$ divides $i - 1$, then $p \cdot v_0^{\nu_p(i-1)} b_{i-1} = 0$ and $i \equiv 1 \pmod{p}$ so that $p \cdot b_i = 0$ too.

  If $p$ does not divide $i-1$ and $i$ is not a power of $p$, we can see from the description of $\THH_*(\BP\langle 1\rangle)$ in Theorem~\ref{thm:AHL} that the order $k$ for multiplication by $p$ of $v_0^{\nu_p(i-1)} b_{i-1} = b_{i -1}$ is such that $k -1$ is the number of digits that are $p - 1$ on the right of the base-$p$ writing of $i - 1$. Then $k-1$ is also the number of carries when adding one to $i -1$, and thus also the $p$-adic valuation of $i$, so that $k$ is the order for multiplication by $p$ of $b_i$.

  If $i$ is a power of $p$, there is one less extension in the tower above $b_{i-1}$, so that the order of $b_i$ is one more than that of $b_{i-1}$.
\end{proof}
We are now prepared to give our nicer description of the kernel and cokernel of the $d_1$-differential. 
\begin{proposition} \label{prop:specseqtorsion}
  In the spectral sequence
  \[
    \THH_*(\BP\langle 1\rangle) \langle \sigma v_2 \rangle \Rightarrow \THH_*(\BP \langle 2 \rangle; \BP\langle 1\rangle)
  \]
  restricted to the torsion elements, the differential $d_1$ satisfies
  \[
    \operatorname{coker} d_1 = \operatorname{coim}(\times v_1^p) \subset \THH_*(\BP\langle 1\rangle)\{ \sigma v_2 \}
  \]
  and
  \[
    \ker d_1 =  \im (\times v_1^p) \oplus M \subset \THH_*(\BP\langle 1\rangle)\{ 1 \}
  \]
  where $\times v_1^p$ is the multiplication by $v_1^p$ map in the proper copy of $\THH_*(\BP\langle 1\rangle)$ and $M$ is the $\BP\langle 1\rangle_*$-submodule of $\THH_*(\BP\langle 1\rangle)$ generated by the $v_0^{k}b_{p^k}$ for all $k \geq 0$.
\end{proposition}

\begin{proof}
  The formula for the differentials of Proposition~\ref{theorem:differentialellcoeff} implies that $\im d_1 \subset \ker (\times v_1^p)$, since for any $i \geq 2$, $d_1(b_i) \in \Z_{(p)}\{ v_0^{\nu_p(i-1)} b_{i-1} \sigma v_2 \} \subset \ker (\times v_1^p)$. But Lemma~\ref{lem:thhell3} implies that
  \[
    d_1(\BP\langle 1\rangle_*\{ b_i \} ) = \BP\langle 1\rangle_* \{ v_0^{\nu_p(i-1)} b_{i-1} \sigma v_2 \}
  \]
  and the description of $\THH_*(\BP\langle 1\rangle)$ implies that
  \[
    \ker (\times v_1^p) = \bigoplus_{i \geq 2} \BP\langle 1\rangle_* \{ v_0^{\nu_p(i-1)} b_{i-1} \}
  \]
  so that
  \[
    \im d_1 = \ker (\times v_1^p).
  \]

  On the other hand, the description of $\THH_*(\BP\langle 1\rangle)$ also implies that a class $\alpha \notin \im (\times v_1^p)$ is a sum of elements of the form $v_1^k v_0^h b_i$ with $0 \leq k < p$. 
  But Lemma~\ref{lem:thhell3} implies that these elements, except for those in $M$, are exactly the sources of the non-zero differentials, so that
  \[
    \ker d_1 = \im (\times v_1^p) \oplus M.
  \]
\end{proof}

\begin{proposition}\label{prop:Einftypage}
  The $\mathrm{E}^\infty$-page of the spectral sequence
  \[
    \THH_*(\BP\langle 1\rangle) \langle \sigma v_2 \rangle \Longrightarrow \THH_*(\BP \langle 2 \rangle; \BP\langle 1\rangle)
  \]
  is the $\Z_{(p)}[v_1]$-module
  \[
    \Z_{(p)}[v_1]\{1, \sigma v_2\} \oplus \mathcal{F}_1 \oplus \Sigma^{2p^2 - 1}\mathcal{F}_1 \oplus \mathcal{I}_1 \oplus M \oplus \Sigma^{2p-1}\mathcal{I}_1
  \]
  where $\mathcal{F}_1$ is the summand of the torsion-free part from Theorem~\ref{thm:AHL}, $\Sigma^{2p^2 - 1} \mathcal{F}_1$ is the copy of $\mathcal{F}_1$ in $\THH_*(\BP \langle 1 \rangle) \{ \sigma v_2 \}$,  $\mathcal{I}_1 := \im (\times v_1^p) \cap T_1^1 \subset \THH_*(\BP\langle 1\rangle)\{ 1 \}$, $\Sigma^{2p-1}\mathcal{I}_1 \cong \operatorname{coim}(\times v_1^p) \cap T_1^1 \subset \THH_*(\BP\langle 1\rangle)\{ \sigma v_2 \}$ and $M \subset \THH_*(\BP \langle 1 \rangle)$ where $T_1^1$ is the torsion summand from Theorem~\ref{thm:AHL} and $\mathrm{im}(\times v_1^p)$, $\mathrm{coim}(\times v_1^p)$ and $M$ are defined in Proposition~\ref{prop:specseqtorsion}.

  All the elements of $\Z_{(p)}[v_1]\{1\} \oplus \Sigma^{2p-1}\mathcal{F}_1 \oplus \mathcal{I}_1 \oplus M$ are in even degrees, and those of $\Z_{(p)}[v_1]\{\sigma v_2\} \oplus \mathcal{F}_1 \oplus \Sigma^{2p^2 -1}\mathcal{I}_1$ are in odd degrees.
\end{proposition}

\begin{proof}
  This description follows from Proposition~\ref{theorem:differentialellcoeff} and Proposition~\ref{prop:specseqtorsion}. The only new claim is that there is an isomorphism of $\Z_{(p)}[v_1]$-modules between $\operatorname{coim}(\times v_1^p)\cap T_1^1$ and $\im(\times v_1^p) \cap T_1^1$ up to a shift, which is given by
  \begin{center}
    \begin{tikzcd}[row sep=tiny]
      \im (\times v_1^p) \cap T_1^1 \rar & \operatorname{coim}(\times v_1^p))\cap T_1^1  \\
      v_1^p b_i \rar[mapsto] & b_i \,.
    \end{tikzcd}
  \end{center}
  It is straightforward from Theorem~\ref{thm:AHL} to check that this is indeed an isomorphism. Since it lowers the degree by $|v_1^p| = 2p^2 - 2p$, and the second copy of $\THH_*(\BP \langle 1 \rangle)\{\sigma v_2 \}$ in the spectral sequence is in degree $|\sigma v_2| = 2p^2 - 1$ higher, in the end the shift is by $2p -1$, so that the torsion of the $\mathrm{E}^\infty$-page is given by $\mathcal{I}_1 \oplus M \oplus \Sigma^{2p-1}\mathcal{I}_1$.  
\end{proof}

%% file: extensions.tex
\subsection{Computing the extensions}\label{sec:extensions}

Between the $\mathrm{E}^2= \mathrm{E}^\infty$-page of the one step spectral sequence
\[
  \THH_*(\BP \langle 1 \rangle)\langle \sigma v_2 \rangle \Longrightarrow \THH_*(\BP \langle 2 \rangle ; \BP \langle 1 \rangle)
\]
computed in Proposition~\ref{prop:Einftypage} and the target module $\THH_*(\BP \langle 2 \rangle ; \BP \langle 1 \rangle)$, we have to solve extension problems given by the short exact sequence of $\BP \langle 1 \rangle _* \cong \Z_{(p)}[v_1]$-modules:
\begin{multline*}
  0 \to \Z_{(p)}[v_1]\{ \sigma v_2\} \oplus \Sigma^{2p^2 -1} \mathcal{F}_1 \oplus \Sigma^{2p-1} \mathcal{I}_1 \to \\
  \THH_*(\BP \langle 2 \rangle; \BP \langle 1 \rangle) \to \Z_{(p)}[v_1]\{ 1\} \oplus \mathcal{F}_1 \oplus \mathcal{I}_1  \oplus M \to 0.
\end{multline*}

Since $\Z_{(p)}[v_1]$ is concentrated in even degrees, it suffices to resolve the extension 
\begin{equation} \label{eq:evenextension}
  0 \to \Sigma^{2p^2 -1} \mathcal{F}_1 \to  \THH_{2*}(\BP \langle 2 \rangle; \BP \langle 1 \rangle) \to \Z_{(p)}[v_1]\{ 1\} \oplus \mathcal{I}_1  \oplus M \to 0
\end{equation}
in even degrees and the extension
\begin{equation}
  \label{eq:oddextension}
  0 \to \Z_{(p)}[v_1]\{ \sigma v_2\} \oplus \Sigma^{2p-1} \mathcal{I}_1 \to \THH_{2*+1}(\BP \langle 2 \rangle; \BP \langle 1 \rangle) \to \mathcal{F}_1 \to 0
\end{equation}
in odd degrees by Proposition~\ref{prop:Einftypage}.

We first determine the $\mathbb{Z}_{(p)}$-module extensions. 
\begin{lemma} \label{lem:onlyextensions}
  The only possible $\mathbb{Z}_{(p)}$-module extensions are of the form
  \[
    p \cdot v_1^h b_1 = v_1^h \lambda_1 \sigma v_2
  \]
  and
  \[
    p \cdot v_1^h v_0^k b_{p^k} = v_1^{p^{k+1} - p + h}v_0^{k-1} a_{p^{k-1}} \sigma v_2
  \]
  for $k \geq 1$ and $0 \leq h \leq p - 1$. 
  Note that for degree reasons there is a unique choice $0 \leq h \leq p - 1$ such that these equalities can hold for each $k\ge 1$. 
\end{lemma}

\begin{proof}
  For a $\Z_{(p)}$-module extension to occur, we need a non-$\sigma v_2$-multiple torsion element (right side of the short exact sequences \eqref{eq:evenextension} and \eqref{eq:oddextension}) and a $\sigma v_2$-multiple element (left side of the short exact sequence \eqref{eq:evenextension} and \eqref{eq:oddextension}). The only torsion on the right side is in even degrees so it is contained in the right side of the short exact sequence \eqref{eq:evenextension}, thus the relation involves an element on the left side which is contained in $\Sigma^{2p^2 -1} \mathcal{F}_1$ and this is torsion-free. Such an element is of the form
  \begin{equation} \label{eq:torsionliftform}
    v_1^h p^k \lambda_1 \sigma v_2 \text{ or } v_1^h p^k v_0^n a_{p^n} \sigma v_2.
  \end{equation}

  Let $\alpha \in \mathcal{I}_1 \oplus M$ be a torsion element in the $E^\infty$ page, such that $p \cdot \alpha = 0$. Let $\bar{\alpha} \in \THH_{2*}(\BP \langle 2 \rangle; \BP \langle 1 \rangle)$ be a lift of $\alpha$. Write $m$ for the smallest integer such that $v_1^m \alpha = 0$. Let $\beta \sigma v_2 \in \Sigma^{2p^2 - 1}\mathcal{F}_1$ be of the form (\ref{eq:torsionliftform}). Write $\tilde{\beta} \sigma v_2$ for its image in $\THH_{2*}(\BP \langle 2 \rangle ; \BP \langle 1 \rangle)$. Suppose that $\beta \sigma v_2$ is divisible by $p$, with $p \cdot \beta' \sigma v_2 = \beta \sigma v_2$, and write $\tilde{\beta'} \sigma v_2$ for the image of $\beta' \sigma v_2$, so that $p \cdot \tilde{\beta'} \sigma v_2 = \tilde{\beta} \sigma v_2$.

  If the extension $p \cdot \bar{\alpha} = \tilde{\beta} \sigma v_2$ occurs, then in $\THH_{2*}(\BP \langle 2 \rangle ; \BP \langle 1 \rangle)$ we must have
  \[
    p v_1^m \bar{\alpha} = v_1^m p \bar{\alpha} = v_1^m \tilde{\beta} \sigma v_2 \neq 0
  \]
  so that $v_1^m \bar{\alpha} \neq 0$. However, the image of $v_1^m \bar{\alpha}$ in $\mathcal{I}_1 \oplus M$ is $v_1^m \alpha = 0$, so that $v_1^m \bar{\alpha}$ must lift to $\Sigma^{2p^2 -1}\mathcal{F}_1$. Since generators $\Sigma^{2p^2 -1}\mathcal{F}_1$ are unique up to a unit, the only possible lift is $v_1^m \beta' \sigma v_2$, which implies that in $\THH_{2*}(\BP \langle 2 \rangle ; \BP \langle 1 \rangle)$ we have an equality
  \[
    v_1^m \bar{\alpha} = v_1^m \tilde{\beta'} \sigma v_2.
  \]
  Let
  \[
    \hat{\alpha} = \bar{\alpha} - \tilde{\beta'} \sigma v_2
  \]
  be a new lift for $\alpha$. We have $p \cdot \hat{\alpha} = 0$ and $v_1^m \hat{\alpha} = 0$, so that we have found a lift that has the same properties as $\alpha \in \mathrm{E}^\infty$.

  Thus, it suffices to consider the extensions involving classes of the form  $\beta \sigma v_2$ that are not divisible by $p$. However, it must still be the case that $v_1^m \bar{\alpha} \neq 0$, so that $v_1^m \beta \sigma v_2$ must be divisible by $p$ in $\Sigma^{2p^2 -1}\mathcal{F}_1$.

  Recall from Theorem~\ref{thm:AHL} the relations in the torsion-free part of $\THH_*(\BP \langle 1 \rangle)$ 
  \[
    \begin{gathered}
      p\cdot a_1 = v_1^p \lambda_1 \\
      p\cdot v_0^n a_{p^n} = v_1^{p^{n+1}}v_0^{n-1} a_{p^{n-1}}
    \end{gathered}
  \] 
  whose degrees in $\Sigma^{2p^2 -1}\mathcal{F}_1$ are $|v_0^n a_{p^n}\sigma v_2| = 2(p^n + 1)p^2 - 2$. The only suitable $\alpha$ in the $\mathrm{E}^\infty$-page with $v_1^m \alpha$ of the right degree are the elements of $M$, that is the classes~$v_0^k b_{p^k}$, since in general we have
  \[
    |v_1^{p^{\nu_p(i) - h + 1 } + \dots + p} v_0^h b_i| = 2(i + p^{\nu_p(i) - h})p^2 - 2 \,.
  \]
\end{proof}

\begin{proposition} \label{prop:zpextensions}
  The $\mathbb{Z}_{(p)}$-module extensions of Lemma~\ref{lem:onlyextensions} all occur. 
\end{proposition}

\begin{proof}
  If these extensions do not occur, we have elements $v_0^k b_{p^k}$ in $\THH_*(\BP \langle 2 \rangle ; \BP \langle 1 \rangle)$ such that
  \[
    v_1^p v_0^k b_{p^k} = 0
  \]
  so that there must be a differential in the spectral sequence
  \[
    \THH_*(\BP \langle 2 \rangle ; \Z_{(p)})[v_1] \Longrightarrow \THH_*(\BP \langle 2 \rangle ; \BP \langle 1 \rangle)
  \]
whose source must be in degree
  \[
    |v_1^p v_0^k b_{p^k}| + 1 = 2(p^k + 1)p^2 - 1.
  \]
  However, by Theorem~\ref{thm:BP2-coeff} there is no torsion element in $\THH_*(\BP \langle 2 \rangle; \Z_{(p)})$ in that degree, since they are in degrees
  \[
    2ip^3 - 1 ,\, 2ip^3 - 1 + 2p - 1,\, 2ip^3 - 1 + 2p^2 - 1 \text{ and } 2ip^3 - 1 + 2p^2 - 1 + 2p - 1 \,.
  \]
\end{proof}

Let us now consider the even degree part given by the short exact sequence \eqref{eq:evenextension}. We claim that it is enough to solve the extension problem over $\Z_{(p)}$ to recover the $\Z_{(p)}[v_1]$-module structure.

\begin{proposition} \label{prop:evenextension}
  The even part of $\THH_*(\BP \langle 2 \rangle; \BP \langle 1 \rangle)$ is given as a $\Z_{(p)}[v_1]$-module by
  \[
    \THH_{2*}(\BP \langle 2 \rangle; \BP \langle 1 \rangle) \cong \Z_{(p)}[v_1]\{1\} \oplus \Sigma^{2p -1} \mathcal{F}_{1, \geq 2p^2 - 1} \oplus \mathcal{I}_{1}
  \]
  where $\mathcal{F}_{1, \geq 2p^2 -1}$ is the submodule of $\mathcal{F}_1$ whose elements are in degree $\geq 2p^2 -1$.
\end{proposition}
\begin{proof}
  Since $\Z_{(p)}[v_1]\{1\}$ is free over $\Z_{(p)}[v_1]$, it is a summand of $\THH_{2*}(\BP \langle 2 \rangle; \BP \langle 1 \rangle)$ as a $\Z_{(p)}[v_1]$-module. We first show that $\mathcal{I}_1$ is also split.

  Let $a, b \in \mathcal{I}_1$ be such that $p^k a = 0$, $p^h b = 0$ and $v_1 a = b$ (as a consequence of this last relation, $k \geq h$). Since, from Lemma~\ref{lem:onlyextensions}, $a$ and $b$ cannot participate in $\Z_{(p)}$-module extensions, we know that there exist lifts $\bar{a}, \bar{b}$ in $\THH_{2*}(\BP \langle 2 \rangle ; \BP \langle 1 \rangle)$ such that $p^k \bar{a} = 0$ and $p^h \bar{b} = 0$. Since we are working in a short exact sequence of $\Z_{(p)}[v_1]$-modules, we know that $v_1 \bar{a} = \bar{b} + \tilde{\theta}$ where $\tilde{\theta}$ is the image of some $\theta \in \Sigma^{2p^2 -1}\mathcal{F}_1$. Then we can compute $p^k \tilde{\theta} = p^k (v_1 \bar{a} - \bar{b}) = 0$, so that $\tilde{\theta} = 0$ since $\Sigma^{2p^2 - 1}\mathcal{F}_1$ is torsion-free, and thus $v_1 \bar{a} = \bar{b}$. Thus, the $\mathbb{Z}_{(p)}[v_1]$-module  $\mathcal{I}_1$ splits from $\THH_*(\BP \langle 2 \rangle ; \BP \langle 1 \rangle )$.

  Lemma~\ref{lem:onlyextensions} and Proposition~\ref{prop:zpextensions} determine the $\Z_{(p)}[v_1]$-module structure for the lifts of the elements of $M$ into $\THH_{2*}(\BP \langle 2 \rangle; \BP \langle 1 \rangle)$. 
  We thus have a short exact sequence of $\Z_{(p)}[v_1]$-modules
  \[
    0 \to \Sigma^{2p^2 -1}\mathcal{F}_1 \to N \to M \to 0
  \]
  where $N$ is generated as a $\Z_{(p)}[v_1]$-module by the lifts of the generators $v_0^n b_{p^n},\, n \geq 0$ of $M$. From the description of $\mathcal{F}_1$ in Theorem~\ref{thm:AHL}, we can check that $N$ is isomorphic to $\Sigma^{2p -1}\mathcal{F}_{1, \geq 2p^2 -1}$ as a $\Z_{(p)}[v_1]$-module via the map
  \[
    v_0^n a_{p^n} \mapsto v_0^n b_{p^n},\, n \geq 0
  \]
  where $v_0^n a_{p^n}$ are the $\Z_{(p)}[v_1]$-module generators of $\mathcal{F}_{1,\geq 2p^2 -1}$. That map raises the degree by $2p -1$.   
\end{proof}

We now consider odd degrees. From the presentation of $\mathcal{F}_1$ of Theorem~\ref{thm:AHL}, we determine a $\Z_{(p)}[v_1]$-free resolution of $\mathcal{F}_1$:
\[
  0 \leftarrow \mathcal{F}_1 \leftarrow \Z_{(p)}[v_1] \{ \lambda_1, v_0^na_{p^n},\, n \geq 1 \} \leftarrow \Z_{(p)}[v_1]\{ r_n,\, n \geq 1 \} \leftarrow 0
\]
where $d(r_1) = p a_1 - v_1^p \lambda_1$ and $d(r_n) = p v_0^n a_{p^n} - v_1^{p^{n+1}} v_0^{n-1} a_{p^{n-1}}$ for any $n \geq 2$. For any $\Z_{(p)}[v_1]$-module $\mathcal{M}$, we thus have
\[
  \operatorname{Ext}^1_{\Z_{(p)}[v_1]}(\mathcal{F}_1, \mathcal{M}) = \operatorname{Hom}_{\Z_{(p)}[v_1]}(\Z_{(p)}[v_1]\{ r_n,\, n \geq 1 \},\,\mathcal{M}) / \sim.
\]
Letting $m_n \in \mathcal{M}$ denote the image of $r_n$ in $M$ for such a class of maps, we can view an extension
\[
  0 \to \mathcal{M} \to \mathcal{N} \to \mathcal{F}_1 \to 0
\]
as a set of \emph{compatible} formulas
\[
  m_1 = p a_1 - v_1^p \lambda_1, \quad m_n = p v_0^n a_{p^n} - v_1^{p^{n+1}} v_0^{n-1} a_{p^{n-1}},\, n \geq 2
\]
that hold in $\mathcal{N}$.

To compute the extension given short exact sequence \eqref{eq:oddextension} for the odd part, we thus need to compute all the classes $m_n$.

\begin{proposition}\label{prop:oddextension}
  The class $m_1\in \mathbb{Z}_{(p)}[v_1]\{\sigma v_2\}\oplus \Sigma^{2p-1}\mathcal{I}_1$ that partially determines the extension $\THH_{2*+1}(\BP\langle 2\rangle;\BP\langle 1\rangle)$ as described above is given by
  \begin{align*}
    m_1&\dot{=} \sigma v_2.
  \end{align*}
\end{proposition}

\begin{proof}
  Remember that the generators of $\mathcal{F}_1$ are $\lambda_1$ and $v_0^n a_{p^n},\, n \geq 0$ in degrees respectively $2p -1$ and $2p^{n+2} - 1$. We can determine from Theorem~\ref{thm:BP2-coeff} that the image modulo $v_1$ of $\lambda_1$ is $\lambda_1$.  Thus a lift of $a_1$ into $\THH_*(\BP \langle 2 \rangle; \BP \langle 1 \rangle)$ has a $p$-tower which is never divisible by $v_1$, so that there is a non trivial extension, with the only possible formula being
  \[
    p a_1 - v_1^p \lambda_1 \dot{=} \sigma v_2
  \]
  up to a unit in $\THH_*(\BP \langle 2 \rangle; \BP \langle 1 \rangle)$.
\end{proof}
Unfortunately, we could not completely determine the remaining extensions with our current methods, but we conjecture the following. 
\begin{conjecture}\label{conj:mn=0}
The classes $m_n$ satisfy $m_n=0$ for $n\ge 2$. 
\end{conjecture}
The formulas in $\THH(\BP)$ computed by Rognes~\cite{Rog20} seem to support this conjecture.

We conclude with the following theorem. 

\begin{theorem}\label{thm:main-v2}
Suppose $\BP\langle 2\rangle$ satisfies Running Assumption~\ref{runningassumption}. 
The topological Hochschild homology of $\BP\langle 2\rangle$ is a direct sum
  \[
    \THH_*(\BP \langle 2 \rangle ; \BP \langle 1 \rangle) \cong \Z_{(p)}[v_1] \langle \lambda_1 \rangle \oplus \mathcal{I}_1 \oplus \mathcal{F}_{1,\geq 2p^2 -1} \oplus \Sigma^{2p - 1} (\mathcal{I}_1 \oplus \mathcal{F}_{1, \geq 2p^2 - 1}) \,,
  \]
  as a $\mathbb{Z}_{(p)}$-module, where $\mathcal{I}_1$ is the $\Z_{(p)}[v_1]$-submodule of $\THH_*(\BP \langle 1 \rangle)$ containing the torsion elements that are in the image of the multiplication by $v_1^p$ map,  $\mathcal{F}_1$ denotes the submodule of the torsion free part of $\THH_*(\BP \langle 1 \rangle)$ of elements not of the form $v_1^k \cdot 1$ for some $k \geq 0$, and $\mathcal{F}_{1, \geq 2p^2 -1}$ is the submodule of $\mathcal{F}_1$ whose elements are in degree $\geq 2p^2 -1$. Assuming Conjecture~\ref{conj:mn=0}, this also holds as $\mathbb{Z}_{(p)}[v_1]$-modules. 
\end{theorem}

\begin{proof}[Proof of Theorem~\ref{thm:main-v2}]
This follows from Proposition~\ref{prop:evenextension} and Proposition~\ref{prop:oddextension}. Since $v_1^p \lambda_1 = p a_1 + \epsilon \sigma v_2$ where $\epsilon$ is a unit, the $v_1$-tower above $\lambda_1$ is never divisible by $p$ and we can split it off of $\mathcal{F}_1$.
\end{proof}

\begin{remark}
As $\mathbb{Z}_{(p)}$-modules, or assuming Conjecture~\ref{conj:mn=0} as $\mathbb{Z}_{(p)}[v_1]$-modules, we can derive an explicit description of $\mathcal{I}_1$  and $\mathcal{F}_1$ from $\THH_*(\BP \langle 1 \rangle)$. We write $c_i$ for the lift of the class $v_1^p b_i$ of the spectral sequence.
Then
  $\mathcal{I}_1$ is a quotient of the $\Z_{(p)}[v_1]$-module
  \[
    \Z_{(p)}[v_1]\{v_0^h c_{\alpha p^n
    },\, n\geq 1,\, \alpha\geq 1,\, p\nmid \alpha,\, h \geq 0\}
  \]
  by the relations:
  \begin{itemize}
  \item $v_0^{h} c_{\alpha p^n} = 0$ for any $\alpha \geq 1$ and $n \geq 1$, $\alpha$ not divisible by $p$, and $h \geq n$.
  \item $v_1^{p^{n-h+1}+p^{n-h}+\dots+p^2 }\cdot v_0^h c_{\alpha p^n} = 0$ for any $\alpha \geq 1$ and $n \geq 1$, $\alpha$ not divisible by $p$ and  $0 \leq h \leq n - 1$.
  \item $p\cdot c_{(\beta p+p-1)p^n} = v_0 c_{(\beta p+p-1)p^n}+v_1^{p^{n+2}}v_0^{\nu_p(\beta)} c_{\beta p^{n+1}}$ for any $\beta \geq 1$ and $n \geq 1$.
  \item $p\cdot v_0^h c_{\alpha p^n} = v_0^{h+1} c_{\alpha p^n}$ for any $\alpha \geq 1$, $n \geq 1$, $\alpha$ not divisible by $p$, and any $1 \leq h \leq n$, or $h=0$ not in the previous case.
  \end{itemize}

  $\mathcal{F}_1$ is a quotient of the $\Z_{(p)}[v_1]$-module
  \[
    \Z_{(p)}[v_1]\{\lambda_1,\, v_0^n a_{p^n},\,n\geq 0\} 
  \]
    by the relations:
  \begin{itemize}
  \item $p\cdot a_1 = v_1^p \lambda_1$.
  \item $p\cdot v_0^n a_{p^n} = v_1^{p^{n+1}}v_0^{n-1} a_{p^{n-1}}$ for any $n\geq 1$. 
  \end{itemize}
  $\mathcal{F}_{1,\geq 2p^2 -1}$ is obtained from $\mathcal{F}_1$ by removing the generator $\lambda_1$.

  $\nu_p$ is the $p$-adic valuation, and the degrees are
  \[
    \begin{gathered}
      |a_i| = 2ip^2 - 1 \,, \\
      |c_i| = 2ip^2 - 1 + 2p^2 - 1 \,.
    \end{gathered}
  \]
\end{remark}

\begin{remark}
We write $\overline{\THH}(R,M)$ for the cofiber of the map $M\to \THH(R,M)$ and when $R$ is an $\mathbb{E}_2$-ring and $M$ is an $\mathbb{E}_1$-$R$-algebra, we regard $\pi_*\overline{\THH}(R,M)$ as an ideal in $\pi_*\THH(R,M)$.
  We can describe
  \[
    \THH_*(\BP \langle 1 \rangle; \BP \langle 0 \rangle) = \THH_*(\ell; \Z_{(p)}) \cong \mathbb{Z}_{(p)}\langle \lambda_1\rangle \otimes (\BP\langle 0\rangle_*\oplus\mathcal{Z}_1)
  \]
  where $\mathcal{Z}_1$ consists of the torsion elements of $\overline{\THH}_*(\BP \langle 0 \rangle)$ divisible by $p = v_0^1$. 

  We have just described (we add an index here to disambiguate)
  \[
    \overline{\THH}_*(\BP \langle 2 \rangle; \BP \langle 1 \rangle)\cong  \mathbb{Z}_{(p)}\langle \lambda_1 \rangle\otimes  (\BP\langle 1\rangle_*\oplus\mathcal{Z}_2) \,,
  \]
  at least as $\mathbb{Z}_{(p)}$-modules, where $\mathcal{Z}_2=\mathcal{I}_1\oplus \mathcal{F}_{1,\ge 2p^2-1}$; the torsion $\mathcal{I}_1$ consists the torsion elements of $\overline{\THH}_*(\BP \langle 1 \rangle)$ divisible by $v_1^p$ and $|\lambda_1|=2p-1$; the torsion-free part $\mathcal{F}_{1, \geq 2p^2 -1}$ is the torsion-free part of $\overline{\THH}_*(\BP\langle 1\rangle)$ in degree above $2p^2 -1$.

  Going down on more height, we can describe
  \[
    \THH_*(\BP \langle 0 \rangle; \BP \langle -1 \rangle) \cong \THH_*(H \Z_{(p)}; H\F_p) \cong \mathbb{Z}_{(p)}\langle \lambda_1 \rangle \otimes( \BP \langle -1 \rangle_* \oplus \mathcal{Z}_0)
  \]
  where $\mathcal{Z}_0$ is the image of the $p$-th power map of the torsion in $\overline{\THH}(\mathbb{F}_p)$. Here there is no torsion free part. 
  
  We speculate that for all $n \geq 1$, we have
  \[
    \THH_*(\BP \langle n \rangle; \BP \langle n-1 \rangle) \cong \Z_{(p)}\langle \lambda_1 \rangle \otimes (\BP\langle n\rangle_*\oplus \mathcal{Z}_n )
  \]
  where in $\mathcal{Z}_n$ the torsion elements are those of $\overline{\THH}_*(\BP \langle n-1 \rangle)$ divisible by $v_{n-1}^{p^{n-1}}$, and the torsion-free part is the torsion-free part of $\THH_*(\BP \langle n -1 \rangle)$ truncated in height above $2p^n - 1$.

  Moreover, the isomorphism between $\mathcal{Z}_1$ and $\Sigma^{2p-1} \mathcal{Z}_1$ is obtained by some operation  denoted $v_0^{-1} \sigma v_1$, the isomorphism between $\mathcal{Z}_2$ and $\Sigma^{2p-1} \mathcal{Z}_2$ is obtained by some operation denoted $v_1^{-p} \sigma v_2$, so that the isomorphism between $\mathcal{Z}_n$ and $\Sigma^{2p-1} \mathcal{Z}_n$ could be obtained by the operation $v_{n-1}^{-p^{n-1}} \sigma v_n$ in the Brun spectral sequence of Proposition~\ref{prop:brunss}.
\end{remark}

\section{Truncated Brown--Peterson spectra are not Thom spectra}\label{notThom}

In this section, we implicitly $2$-localize and write $\ku=\ell=\BP\langle 1\rangle$ and $\BP\langle n\rangle$ denotes an $\mathbb{E}_3$-$\MU$-algebra form of the $n$-th truncated Brown--Peterson spectrum at the prime $p=2$. 
The goal of this section is to prove Theorem~\ref{thm:main-thom}. 

First we need some notation. Given an $\mathbb{E}_2$-ring $A$ and an $\mathbb{E}_1$-$R$-algebra $B$ , we write 
 $\overline{\mathrm{THH}}(A;B)$ for the cofiber of the unit map 
$B\to \mathrm{THH}(A;B)$ in $B$-modules. The only input for our proof is the following computational result. 

\begin{proposition}\label{prop-computational-input}
We compute that $\pi_*\overline{\mathrm{THH}}(\BP\langle n\rangle;\ku)$  is generated as a $ku_*$-module in degrees $\le 8$ by $\lambda_1$, $a_1$ and $\sigma v_2$ in degrees $3$, $7$, $7$ respectively and there is a relation $v_1^2 \lambda_1 = 2 a_1 + \alpha \sigma v_2$ for some unit $\alpha \in \Z_{(2)}$.
\end{proposition}

\begin{proof}
We consider the fiber sequence  
\[ \THH(\ku;F) \to  \THH(\BP\langle n\rangle ; \ku )\to 
\THH(\ku) 
\]
where $F$ is the fiber of the map 
\[ \ku \otimes_{\BP\langle n \rangle }\ku \to \ku \]
of $\ku\otimes \ku^{\textup{op}}$-modules. 
The map $\pi_*\ku\to \pi_*\THH(\ku)$ is an isomorphism in degrees $*\le 2$ by Theorem~\ref{thm:AHL}. Since we also know that  
\[ 
\pi_k F = \begin{cases} 0 & \text{ if } k<7\text{ or } k = 8,\\
           \mathbb{Z}_{(2)} &\text{ if } k=7  
\end{cases}
\]
we can conclude that the inclusion of $0$-simplices $F\to \THH(\ku;F)$ is an isomorphism in degrees $\le 8$. We write $\sigma v_2$ for the image of the generator of $\THH_{7}(\ku;F)=\mathbb{Z}_{(2)}$ in $\THH_7(\BP\langle n\rangle ;\ku)$. Note that $\THH_7(\ku)=\mathbb{Z}_{(2)}\{a_1\}$ and that $2 a_1 = v_1^2 \lambda_1$ in $\THH_7(\ku)$ so that any lift of $a_1$ and $\lambda_1$ in $\THH_*(\BP \langle n \rangle; \ku)$ satisfy a relation of the form $v_1^2 \lambda_1 = 2 a_1 + \alpha \sigma v_2$. To determine that $\alpha$ is a unit, it suffices to determine the relation mod $v_1$ where we can show that there is a nontrivial extension 
\[ 0\to \mathbb{Z}_{(2)}\{\sigma v_2 \} \to \mathbb{Z}_{(2)}\{\lambda_2 \}\to \mathbb{F}_{2}\{\lambda_1\frac{v_1^2}{2}\}\to 0 \] 
in this degree by \cite[Theorem~3.18]{ACH24}. Note that the relation $2\lambda_2=\sigma v_2$ mod $v_1$ is also consistent with~\cite[Theorem~1.1]{Rog20}. 
\end{proof}

\begin{proof}[Proof of Theorem~\ref{thm:main-thom}] 
Suppose $\BP\langle n\rangle$ is the Thom spectrum of a $2$-fold loop map at the prime $2$. Then 
\[ \THH(\BP\langle n\rangle;\mathrm{ku})\simeq \mathrm{ku} \otimes X_{+}\simeq \mathrm{ku} \otimes_{\mathrm{ko}}(\mathrm{ko} \otimes X_+)
\]
by~\cite[Theorem~2]{BCS10} together with the fact that $\eta$ has trivial image under the Hurewicz map $\mathbb{S}\to \BP\langle n\rangle$, see the proof of~\cite[Theorem~2]{Bea17}. We argue that no such $Y=\mathrm{ko}\otimes X_{+}$ exists by the same argument as~\cite{AHL09}. This relies on Proposition~\ref{prop-computational-input}. We include the details for completeness since the argument from~\cite{AHL09} needs to be adapted slightly. 

Suppose a cellular $ko$-module $Y$ existed. If it did, it would have cells beginning in degrees $3$, $7$, $7$ and $8$ by Proposition~\ref{prop-computational-input}. Since $\pi_3\mathrm{ku}=0$, we know the first attaching map of $X$ is zero and therefore the first attaching map of $Y$ is also trivial. We therefore know that $Y$ begins with 
\[ \Sigma^{3}\mathrm{ko}\vee \Sigma^{7}\mathrm{ko}\vee \Sigma^{7}\mathrm{ko}\to (\Sigma^{3}\mathrm{ko}\vee \Sigma^{7}\mathrm{ko}\vee \Sigma^{7}\mathrm{ko})\cup_{\phi} C\Sigma^{7}\mathrm{ko})\to Y
\]
The attaching map $\phi$ is an element in 
\begin{equation}\label{pi7ko}\pi_{7} (\Sigma^{3}\mathrm{ko}\vee \Sigma^{7}\mathrm{ko}\vee \Sigma^{7}\mathrm{ko})=\mathbb{Z}\{v_1^2\}\oplus \mathbb{Z}\{1\}\oplus \mathbb{Z}\{1\} \,.
\end{equation}
By the relation $v_1^2\lambda_1=2a_1 + \alpha \sigma v_2$, the goal would be to consider the class 
\[
(v_1^2,-2,-\alpha) \in \pi_{7} (\Sigma^{3}\mathrm{ku}\vee \Sigma^{7}\mathrm{ku}\vee \Sigma^{7}\mathrm{ku})=\mathbb{Z}\{v_1^2\}\oplus \mathbb{Z}\{1\}\oplus \mathbb{Z}\{1\} 
\]
and lift it to \eqref{pi7ko}. However, the map 
\[
\pi_7(\Sigma^{3}\mathrm{ko}\vee \Sigma^{7}\mathrm{ko}\vee \Sigma^{7}\mathrm{ko})  \to 
\pi_7(\Sigma^{3}\mathrm{ku}\vee \Sigma^{7}\mathrm{ku}\vee \Sigma^{7}\mathrm{ku}) 
\]
is induced by the complexification map $\mathrm{ko}\to \mathrm{ku}$, so it is multiplication by 2 on the $\Sigma^3$ summand and identity on the $\Sigma^7$ summands, and therefore $(v_1^2,-2,-\alpha)$ is not in the image, leading to a contradiction. 
\end{proof}

\begin{remark}
As in the case of~\cite{AHL09}, our argument for proving Theorem~\ref{thm:main-thom} requires that we work at the prime $p=2$. 
\end{remark}

\begin{remark}
We have proven that $\BP\langle n\rangle$ at $p=2$ is not a Thom spectrum of a $2$-fold loop map over the sphere spectrum, however it may still be a Thom spectrum in the category of modules over a different spectrum in the sense of~\cite{ABGHR14}. For example, it has been conjectured that it is the Thom spectrum over certain Ravenel spectra~\cite[Conjecture~1.1.3]{Dev24} and it has been proven that it is the Thom spectrum in the category of modules over a polynomial algebra over complex cobordism in~\cite{Qui25}. 
\end{remark}